\documentclass{amsart}
\usepackage{amsmath, amssymb, amsthm}
\usepackage{url}
\usepackage{mathtools}

\def \N {\mathbb{N}}
\def \Z {\mathbb{Z}}

\def \A {\mathcal{A}}

\def \H {\mathcal{H}}

\def \T {\mathcal{T}}

\def \ol {\overline}

\def \WP {\operatorname{WP}}

\def \rev {\mathrm{rev}}

\def \C {\mathfrak{C}}

\DeclarePairedDelimiter{\set}{\{}{\}}
\DeclarePairedDelimiterX{\gset}[2]{\{}{\}}{\,#1:#2\,}

\newtheorem{theorem}{Theorem}

\newtheorem{example}[theorem]{Example}
\newtheorem{question}[theorem]{Question}

\begin{document}

\title{Word Problem Languages for Completely Regular Semigroups}

\author{Tara Brough}

\address{Centro de Matem\'{a}tica e Aplica\c{c}\~{o}es, Faculdade de Ci\^{e}ncias e Tecnologia \\
Universidade Nova de Lisboa, 2829--516 Caparica, Portugal}
\email{t.r.brough@fct.unl.pt (institutional); tarabrough@gmail.com (permanent and preferred)}

\maketitle

\begin{abstract}
Motivated by the question of which completely regular semigroups have 
context-free word problem, we show that for certain classes of languages $\C$
(including context-free), every completely regular semigroup that is a union 
of finitely many finitely generated groups with word problem in $\C$ also has 
word problem in $\C$.  We give an example to show that not all completely regular
semigroups with context-free word problem can be so constructed.
\end{abstract}

\section{Introduction}

The word problem of a semigroup is, informally, the problem of determining whether two words
over a given generating set represent the same element of the semigroup.
While in general the word problem is undecidable for semigroups \cite{post} and 
even for finitely-presented groups \cite{boone,nov}, there has been much research into 
semigroups (and especially groups) for which the word problem is in some sense `easily'
decidable, in that it can be solved by an automaton less powerful than a Turing machine.
Some authors approach this from a time- or space-complexity perspective
(see for example \cite{birget_time, kambites_adequate, lohrey_inverse}), 
and others from a language-theory perspective (for a survey up to 2014 see
\cite[Sections~5 and~6]{jones_formal} and for more recent developments
\cite{brough_inverse, bcp_contextfree, ho_multiple}).
In this paper we take the language-theoretic approach.

As a language, the \emph{word problem} of a semigroup $S$ with respect to a finite generating
set $A$ is commonly defined as the set \[\WP(S,A) = \gset{u\#v^\rev}{u,v\in A^+, u =_S v},\] where 
$v^\rev$ denotes the reverse of $v$ and $\#$ is a new symbol not in $A$.
This definition is called the \emph{unfolded word problem} by the author and 
Cain \cite{bc_hierarchy}, to distinguish it from the \emph{two-tape word problem}, which 
is the relation $\gset{(u,v)}{u,v\in A^+, u=_S v}$.
So far the study of semigroup word problems as languages has mainly been for 
context-free languages in the unfolded definition \cite{hoffmann_contextfree, bcp_contextfree},
and for rational relations in the two-tape definition \cite{pfeiffer_iteration, brough_inverserational}.
The relationship between the two definitions was studied in \cite{bc_hierarchy}.
If $\WP(S,A)\in \C$ for some finite generating set $A$, and $\C$ is a language
class closed under inverse homomorphism, then $\WP(S,B)\in \C$ for all finite generating sets 
$B$ of $S$ and we say that \emph{$S$ is $U(\C)$}.

Semigroups with context-free word problem were studied initially in \cite{hoffmann_contextfree},
with one of the main results being the characterisation of the completely simple semigroups with 
context-free word problem as (finitely generated) Rees matrix semigroups over virtually free groups.
This is a fairly straightforward consequence of two major results: the characterisation by Rees
of the completely simple semigroups as what are now called Rees matrix semigroups 
over groups \cite{rees_matrix}, and Muller and Schupp's characterisation of groups with context-free word problem as the finitely generated virtually free groups \cite{muller_groups}.  
In a further study of semigroups with 
context-free word problem, the author, Cain and Pfeiffer showed more generally that if $\C$ is a class of languages closed under inverse gsm mappings and intersection
with regular languages, then a finitely generated Rees matrix semigroup 
over a semigroup $S$ is $U(\C)$ if and only if $S$ is $U(\C)$ \cite[Theorem~7]{bcp_contextfree}.
They also showed that if furthermore $\C$ is closed under union, then 
a \emph{strong} semilattice of semigroups 
$S_1,\ldots,S_n$ is $U(\C)$ if and only if each $S_i$ is $U(\C)$ 
\cite[Theorem~6]{bcp_contextfree}.

These results led to a question by M\'{a}rio Branco \cite{branco_private} 
on which completely regular semigroups have 
context-free word problem.
A semigroup is \emph{completely regular} if it can be expressed as a disjoint union of groups.
Clifford showed that a semigroup is completely regular if and only if it is isomorphic to a semilattice of 
completely simple semigroups \cite{clifford_relativeinverse}.  Thus the results of the author, Cain 
and Pfeiffer naturally lead to questions about which completely regular semigroups are $U(\C)$ for language classes $\C$ with all the closure properties mentioned.

Previous research on word problems of completely regular semigroups has concentrated 
on determining decidability of the word problem in certain varieties of completely regular semigroups
\cite{gerhard_varieties} and solving the word problem for free completely regular semigroups
\cite{gerhard_free, kadourek_free}.  (Free completely regular semigroups are not finitely generated as semigroups, though they are as (2,1,0)-algebras, and so a language-theoretic study of their word problem would require a different definition from the one used here.)
The present paper appears to be the first to consider these word problems from a language-theoretic 
perspective.

\section{Preliminaries}

\subsection{Structure of completely regular semigroups}

The semilattice structure of completely regular semigroups was studied in detail
in the 1960s and '70s by Lallement and Petrich \cite{lallement_demi, petrich_structure}.  
A central concept in their work is the notion of a bitranslation:
Let $S$ be a semigroup.
A \emph{bitranslation} of $S$ is a pair $(\lambda,\rho)$ such that
$\lambda(xy) = (\lambda x)y$ and
$(xy)\rho = x(y\rho)$ for all $x,y\in S$
(that is, $\lambda$ is a left translation and $\rho$ a right translation), 
and furthermore $x(\lambda y) = (x\rho) y$ for all $x,y\in S$.
The \emph{translational hull} $\Omega(S)$ of $S$ is the set of all bitranslations of $S$.
A bitranslation $(\lambda,\rho)$ is \emph{inner} if for some $a\in S$ we have 
$\lambda x = ax$ and $x\rho = xa$ for all $x\in S$.  
We write $(\lambda,\rho) = (\lambda_a,\rho_a) =  \pi_a$
in this case.
The set of all inner bitranslations of $S$ is $\Pi(S)$.

Lallement showed that the multiplication in a completely regular semigroup
is determined by certain maps from the constituent completely simple 
semigroups to the translational hulls of the completely simple semigroups lower in 
the semilattice, as follows.

\begin{theorem}[{\cite[Th\'{e}or\`{e}me 2.19]{lallement_demi}}]\label{lallement}
Let $Y$ be a semilattice; to each $\alpha\in Y$ associate a completely simple
semigroup $S_\alpha$ and suppose that $S_\alpha\cap S_\beta = \emptyset$
if $\alpha\neq \beta$.
For each pair $\alpha,\beta\in Y$ with $\alpha\geq \beta$,
let $\Phi_{\alpha,\beta}: S_\alpha\rightarrow \Omega(S_\beta)$ be a function satisfying:
\begin{enumerate}
\item $\Phi_{\alpha,\alpha}: a\mapsto \pi_a \quad (a\in S_\alpha)$;
\item $(S_\alpha \Phi_{\alpha,\alpha\beta})(S_\beta \Phi_{\beta,\alpha\beta})
\subseteq \Pi(S_{\alpha\beta})$;
\item If $\alpha\beta>\gamma$ and $a\in S_\alpha$, $b\in S_\beta$, then
\[(a\star b)\Phi_{\alpha\beta,\gamma} = (a\Phi_{\alpha,\gamma})(b\Phi_{\beta,\gamma}),\]
\end{enumerate}

where multiplication in $S$ is defined by
\[ a\star b = [(a\Phi_{\alpha,\alpha\beta})(b\Phi_{\beta,\alpha\beta})]\Phi_{\alpha\beta,\alpha\beta}^{-1}
\quad (a\in S_\alpha, b\in S_\beta).\]
Then $S$ is a completely regular semigroup.  Conversely, every completely regular semigroup
can be so constructed.
\end{theorem}


Upon quoting the above theorem, Petrich \cite{petrich_structure} mentions the following points, 
worth also mentioning here:
``It should be remarked that conditions (i) and (iii) in Theorem 1 imply that each 
$\Phi_{\alpha,\beta}$ is a homomorphism, and since $S_\alpha$ is bisimple, 
we then have that $S_\alpha\Phi_{\alpha,\beta}$ is contained in a ${\mathcal D}$-class of
$\Omega(S_\beta)$.
Also note that condition (ii) appears only in order to make condition (iii) meaningful."

For a set $X$, denote the left and right transformation semigroups of $X$ respectively
by $\T(X)$ and $\T'(X)$.
The translational hull of a Rees matrix semigroup $S = \mathcal{M}[I,G,\Lambda;P]$ can be isomorphically embedded in a direct product $\T(I)\times G^\Lambda\times \T'(\Lambda)$,
where $G^\Lambda$ is the set of all mappings from $\Lambda$ into $G$.
We denote the constant transformations mapping $X$ to $x\in X$ by 
$\tau_x$ in $\T(X)$ and $\tau'_x$ in $\T'(X)$.
The following is extracted from 
\cite[Theorem~2]{petrich_structure}, which we do not quote in full as we do not require all the details.

\begin{theorem}
\label{petrich_transhull}
Let $S = \mathcal{M}[I,G,\Lambda;P]$ be a Rees matrix semigroup.
There exists an isomorphism from $\Omega(S)$ to a subsemigroup of  
$\T(I)\times G^\Lambda\times \T'(\Lambda)$, and for $(i,g,\lambda)\in S$, the image
of $\pi_{(i,g,\lambda)}$ under this isomorphism is $(\tau_i,h,\tau'_\lambda)$ for some 
$h\in G^\Lambda$.
\end{theorem}
\begin{proof}
The isomorphism in question is the inverse of the isomorphism $\chi$ in 
\cite[Theorem~2]{petrich_structure}.
\end{proof}

\subsection{Operations on languages}

The operations used in this paper are union, inverse gsm mapping and 
intersection with regular languages.  
For the definition of inverse gsm mappings, and an overview of which language classes are 
closed under each of these operations, see \cite[Section~11.2]{hopcroft_automata}.
A \emph{full trio} is a class of languages closed under homomorphism, inverse homomorphism 
and intersection with regular languages.  Many well-known full trios, such as the regular, 
context-free, ET0L and indexed languages, are also closed under union and 
inverse gsm mappings and thus satisfy the hypothesis of our main result 
(Theorem~\ref{main}).
An example of a class that is closed under our three operations but is not a full trio is the 
poly-context-free languages (intersections of finitely many context-free languages)
\cite[Section~2.2]{brough_polycf}, which are not closed under homomorphism.

\section{Constructing completely regular semigroups with word problem in $\C$ from 
finitely many groups with word problem in $\C$}

\begin{theorem}\label{main}
Let $S$ be a completely regular semigroup that is a union of finitely many 
finitely generated $\C$-groups, where $\C$ is a class of languages closed 
under union, inverse gsm mappings and intersection 
with regular languages.  Then the (unfolded) word problem of $S$ is in $\C$.
\end{theorem}

\begin{proof}
By the result of Clifford \cite{clifford_relativeinverse}, $S$ can be expressed as a 
finite semilattice of finitely generated completely simple semigroups.
Let $S = \bigcup_{\alpha\in Y} S_\alpha$ be a semilattice of 
completely simple semigroups $S_\alpha$, with $Y$ finite.  
For each $\alpha\in Y$, let $G_\alpha$ be a $\C$-group, $I_\alpha$ and $\Lambda_\alpha$ finite sets  
and $P_\alpha$ a $\Lambda_\alpha\times I_\alpha$ matrix such that
$S_\alpha = {\mathcal M}[G_\alpha; I_\alpha, \Lambda_\alpha; P_\alpha]$.
We have $\WP(S_\alpha,X_\alpha)\in \C$ for all $\alpha\in Y$ by 
\cite[Theorem~2]{hoffmann_contextfree}.
Since $I_\alpha$ and $\Lambda_\alpha$ are finite, each $S_\alpha$ has finitely many 
$\H$-classes $H^\alpha_{i\lambda}$, $i\in I_\alpha$, $\lambda\in \Lambda_\alpha$.
Let $X_\alpha$ be a finite generating set for $S_\alpha$, for each $\alpha\in Y$,
and let $X = \bigcup_{\alpha\in Y} X_\alpha$, which is a finite generating set for $S$.
For $\alpha\in Y$, $i\in I_\alpha$, $\lambda\in \Lambda_\alpha$, define
\[
L^\alpha_{i\lambda} = \gset{ u\# v^\rev }{ u\#v^\rev\in \WP(S,X), \ol{u}, \ol{v}\in H^{\alpha}_{i\lambda}}.
\]
That is, $L^\alpha_{i\lambda}$ is the restriction of $\WP(S,X)$ to the $\H$-class 
$H^\alpha_{i\lambda}$.  Thus $\WP(S,X)$ is the union of the finitely many languages $L^\alpha_{i\lambda}$.

Fix $\alpha\in Y$, $i\in I_\alpha$, $\lambda\in \Lambda_\alpha$ and let $H = H^\alpha_{i\lambda}$,
$L = L^\alpha_{i\lambda}$ and $G = G_\alpha$.  We will show that $L\in \C$.  
Since $\alpha,\lambda,i$ are arbitrary and $\C$ is closed under 
union, this will establish that $\WP(S,X)\in \C$.

We may assume (by change of generators and isomorphism) that the matrix $P_\alpha$ 
is normalised to have the identity $1_G$ of $G$ in every position in the first row and column.
Let $e = (i, 1_G, 1)$ and $f = (1, 1_G, \lambda)$.  Then $e$ and $f$ are idempotents acting 
as left and right identities respectively on $H$.  We may assume (by change of generators
if necessary) that $e,f\in X_\alpha$.

For any $u,v\in X^*$ with $\ol{u},\ol{v}\in H$, we have $u\#v^\rev\in L$ if and only 
if $eu\#fv^\rev\in L$, since $eu = u$ and $vf = v$.   Thus $L$ is the intersection of 
two languages $L_1 = \gset{u\#v}{eu\#fv^\rev\in \WP(S_\alpha,X_\alpha)}$ and 
$L_2 = \gset{u\#v}{\ol{u},\ol{v}\in H}$.  Since $H$ consists precisely of all $(i,g,\lambda)$,
$g\in G$, Theorem~\ref{petrich_transhull} ensures that membership in $L_2$ can 
be computed by a finite automaton, since we only need to perform multiplication in the 
finite transformation semigroups $\T(I)$ and $\T'(\Lambda)$ and check that the results
are the constant functions $\tau_i$ and $\tau'_\lambda$.  Thus $L_2$ is regular and so 
it suffices to show that $L_1\in \C$.

For each $x\in X_\alpha$ and $y\in X_\beta$ for $\beta\geq\alpha$, choose 
$w_{xy}, z_{xy}\in X_\alpha^*$ and $t_{xy}, s_{xy}\in X_\alpha$ such that $xy = w_{xy}t_{xy}$
and $yx = s_{xy} z_{xy}$.
If $\beta = \alpha$ we choose $w_{xy} = z_{xy} = x$ and $t_{xy} = s_{xy} = y$.
Let $W$ be the set of all $w_{xy}$, and define $Z$, $S$, $T$ similarly.
Moreover, let $W_e = \set{w_{ey}}$ and $Z_f = \set{z_{fy}}$.

Define a gsm-mapping $\Phi: X^*\rightarrow X_\alpha^*$ by an automaton $\A$ as follows:
The state set is $X_\alpha\cup X'_\alpha\cup \{\$\}$, where $X'_\alpha = \gset{x'}{x\in X_\alpha}$
is a copy of $X_\alpha$ and $\$$ is the final state.  
The start state is $e$, and for each state $x\in X_\alpha$ and input symbol $y\in X_\beta$
there is a transition to $t_{xy}$ outputting $w_{xy}$.  Each state $x\in X_\alpha$ also has a 
transition on $\#$ to $f'$ with output $x$.  Further, from each $x'\in X'_\alpha$
and $y\in X_\beta$ there is a transition to $s'_{xy}$ with output $z_{xy}^\rev$.
There is also an $\varepsilon$-transition from each $x'$ to the final state $\$$ with output $x$.
The final output on reading $u\#v^\rev$ with $u,v\in X_\alpha^*$ is some $\hat{u}\#\hat{v}^\rev$
with $\hat{u}\in W_e W^* T$ representing $eu$ and $\hat{v}\in S Z^* Z_f$ representing $vf$.
Hence the preimage of $\WP(S_\alpha,X_\alpha)\cap W_e W^* T\# S Z^* Z_f$ in $X^*$
is $L_1 = \gset{ u\#v}{ eu\#fv\in \WP(S_\alpha,X_\alpha)}$, which is in $\C$ since $\C$ is closed
under intersection with regular languages and inverse gsm-mappings.

Hence $L\in \C$ and so $\WP(S,X)$, as the union of finitely many such languages $L$, 
is in $\C$.
\end{proof}

In particular, since the groups with context-free word problem are the virtually free groups
\cite{muller_groups}, a union of finitely many finitely generated virtually free groups has
context-free word problem.

\section{Examples showing the necessity for further work}

The converse of Theorem~\ref{main} does not hold.  Examples abound, but here is one
of the most straightforward.

\begin{example}
Let $S$ be the subsemigroup of $\T(\Z)$ generated by $x: i\mapsto i+1$, 
$x^{-1}: i\mapsto i-1$ and the constant transformations $\tau_j: i\mapsto j$
for each $j\in \Z$.  Then $S$ is a completely regular semigroup with 
context-free word problem, but $S$ cannot be expressed as a 
union of finitely many finitely generated groups.
\end{example}
\begin{proof}
$S$ can be expressed as the union of $\Z = \langle x, x^{-1}\rangle$ and infinitely 
many trivial groups $\langle \tau_i\rangle$, and is thus completely regular.
If $S$ is a union of finitely many finitely generated groups, then it can be expressed
as a finite semilattice of finitely generated completely simple semigroups 
\cite{clifford_relativeinverse}.
It is a semilattice of $\Z = \langle x,x^{-1}\rangle$ and the infinite right-zero 
semigroup $T = \gset{\tau_i}{i\in \Z}$, but since the orbit of any $\tau_i\in T$
under the action of $\Z$ is the whole of $T$, and $T$ is not finitely generated,
it is not possible to decompose $S$ as a semilattice of finitely generated 
semigroups.

The word problem of $S$ with respect to the generating set $\set{x,x^{-1},\tau_0}$
is recognised by a pushdown automaton $\A$ that behaves like the standard automaton 
for $\WP(\Z,\set{x,x^{-1}})$ with the following modifications: There three extra states
$p_0$, $p_1$ and $p_2$, which are used for words containing $\tau_0$; before $\#$ has 
been seen the stack is erased whenever $\tau_0$ is seen, and the first time this 
happens $\A$ moves to state $p_0$, which behaves the same as the initial state
on $\{x,x^{-1},\tau_0\}$.  
If $\A$ is in state $p_0$ on input $\#$, it moves to state $p'$ instead of the usual state
(call it $q_1$).  State $p_1$ acts in the same way as $q_1$ except that it moves to state $p_2$
(which accepts all inputs in $\set{x,x^{-1},\tau_0}$ and is final)
on input $\tau_0$ if and only if the stack is empty, whereas $q_1$ fails on input $\tau_0$.
Thus on input $u\#v^\rev$, $\A$ first records $\ol{u}\in \gset{x^i, \tau_i}{i\in \Z}$ using the 
stack and state, then checks this against $v^\rev$, making sure that $v$ contains 
$\tau_0$ if and only if $u$ does.  Hence $S$ has context-free word problem.
\end{proof}

However, unsurprisingly not every finitely generated semigroup 
that is a union of context-free groups has context-free word problem.

\begin{example}
Let $S$ be the union of the free group $F$ of rank $2$ generated by $\set{x,y}$
and infinitely many idempotents $b_i, i\in \Z$, with multiplication
defined as follows:  $B = \gset{b_i}{i\in \N_0}$ is a left-zero semigroup;
the right action of $F$ on $B$ is the identity; $xb_i = b_{i+1}$ for all $i\in \Z$;
$y$ translates the set $\gset{b_{-2^n}, b_{2^n}}{n\in \N_0}$, ordered by subscript,
to the right; the actions of $x^{-1}$ and $y^{-1}$ are the inverse of the actions of 
$x$ and $y$ respectively.  Then $S$ is a finitely generated union of context-free groups that does
not have context-free word problem.
\end{example}
\begin{proof}
We first establish that the definition of the multiplication is sound.
The multiplication indeed defines left and right actions of $F$ on $B$,
since the actions of $x$ and $y$ are invertible and the actions of 
$x^{-1}$ and $y^{-1}$ respectively are the inverses of the corresponding actions.
Moreover, the fact that $B$ is a left-zero semigroup and the right action of 
$F$ fixes $B$ ensures that $(b_i g) b_j = b_i (g b_j)$
for all $i,j\in \N_0$ and $g\in F$ (that is, the left and right actions of each 
element of $F$ form a bitranslation). 

Free groups and trivial groups are context-free, so $S$ is a union of context-free groups;
and $S$ is finitely generated by $X = \set{x,y,x^{-1},y^{-1},b_0}$.
Let $L = \WP(S,X)\cap x^*b_0\#b_0y^*$.
Then $L = \gset{x^{2^n}b_0\#b_0y^n}{n\in \N_0}$, which is easily shown by the 
pumping lemma not to be context-free.  Since $L$ is the intersection of the word 
problem of $S$ with a regular language, $S$ does not have
context-free word problem.
\end{proof}

The full characterisation of completely regular semigroups with context-free word problem
remains open.  More generally, though probably considerably more challengingly,
the following question remains open.

\begin{question}
Let $\C$ be a class of languages closed under union, inverse gsm mappings
and intersection with regular languages.  When is a semilattice of Rees matrix semigroups
over $U(\C)$ groups a $U(\C)$ semigroup?
\end{question}

(We know that every completely regular $U(\C)$ semigroup is a semilattice of 
Rees matrix semigroups over $U(\C)$ groups by \cite[Theorem~7]{bcp_contextfree}.)

\section*{Acknowledgements}

This research was supported by the Funda\c{c}\~{a}o para a Ci\^{e}ncia e a Tecnologia (Portuguese Foundation for Science and Technology)
through an {\sc FCT} post-doctoral fellowship ({\sc SFRH}/{\sc BPD}/121469/2016) and the projects {\sc UID}/{\sc MAT}/00297/2013 (Centro de Matem\'{a}tica e Aplica\c{c}\~{o}es) and {\sc PTDC}/{\sc MAT-PUR}/31174/2017.

The author is grateful to Alan Cain for helpful discussions, and particularly for the trick 
with the idempotents.

\end{document}